\documentclass{article}
\usepackage{amsfonts}
\usepackage{amssymb,bm}
\usepackage{amstext}
\usepackage{wrapfig}
\usepackage{booktabs} 
\usepackage{amsmath}
\usepackage{xspace}
\usepackage{theorem}
\usepackage{graphicx}
\usepackage{url}
\usepackage{enumitem}
\usepackage{graphics}
\usepackage{colordvi}
\usepackage{colordvi}
\usepackage{subfigure}
\usepackage{xr-hyper}
\usepackage{hyperref}
\usepackage{float}

\newtheorem{theorem}{Theorem}
\newtheorem{definition}{Definition}
\newtheorem{proposition}{Proposition}

\newenvironment{proof}{\par \smallskip{\bf Proof:}}{\hfill\stopproof}
\def\stopproof{\square}
\def\square{\vbox{\hrule height.2pt\hbox{\vrule width.2pt height5pt \kern5pt
\vrule width.2pt} \hrule height.2pt}}

\begin{document}
\title{On a Faster $R$-Linear Convergence Rate of the Barzilai-Borwein Method}
\author{
{Dawei Li} \thanks{Department of ISE, Coordinated Science Laboratory, University of Illinois at Urbana-Champaign, Urbana, IL. \texttt{dawei2@illinois.edu}.}
{\quad \quad  \quad Ruoyu Sun} \thanks{Department of ISE, and affiliated to Coordinated Science Laboratory and Department of ECE, University of Illinois at Urbana-Champaign, Urbana, IL. \texttt{ruoyus@illinois.edu}.} 
\date{\today}
}
\maketitle

\begin{abstract}
    The Barzilai-Borwein (BB) method has demonstrated great empirical success in nonlinear optimization. However, the convergence speed of BB method is 
    not well understood, as the known convergence rate of BB method
     for quadratic problems is much worse than the steepest descent (SD) method.
     Therefore, there is a large discrepancy between theory and practice.
    To shrink this gap,
    we prove that the BB method converges $R$-linearly at a rate of $1-\frac{1}{\kappa}$, where $\kappa$ is the condition number, for strongly convex quadratic problems. In addition,  an example with the theoretical rate of convergence is constructed, indicating the tightness of our bound. 
\end{abstract}

\section{Introduction}\label{sec:intro}

The Barzilai-Borwein (BB) method is a simple and efficient gradient algorithm. It takes the same search direction as the classic steepest descent (SD) method \cite{cauchy1847methode}, with a somewhat delicate stepsize rule.
Since its proposal by Barzilai and Borwein \cite{barzilai1988two}, the BB method has been shown to give competitive performances in global optimization \cite{raydan1997barzilai}, and its extensions and variants have been utilized in large numbers of applications including compressed sensing \cite{figueiredo2007gradient}, image restoration \cite{wang2007projected}, sparse reconstruction \cite{wen2010fast}, signal processing \cite{liu2010coordinated}, matrix factorization \cite{huang2015quadratic}, machine learning \cite{tan2016barzilai} and distributed optimization \cite{gao2019geometric}. \\

Despite the great empirical success of the BB method, the non-monotone behaviors of the algorithm make it difficult to perform theoretical analysis. 
Most theoretical results of the BB method are derived for strictly convex quadratic objective functions. 
The earliest result was presented by Barzilai and Borwein \cite{barzilai1988two}, who provided an $R$-superlinear convergence proof in the $2$-dimensional case. In the general $n$-dimensional case, global convergence was proved by Raydan \cite{raydan1993barzilai}, and an $R$-linear convergence result was given by Dai and Liao \cite{dai2002r}, the proof technique of which was also utilized to show the $R$-linear convergence of other alternative stepsize methods \cite{dai2003alternate, dai2019family}. Dai and Fletcher \cite{dai2005asymptotic} also made an interesting asymptotic analysis, a corollary of which was $R$-superlinear convergence of BB method in the $3$-dimensional case. \\

Although these results demonstrated some good properties
of BB method, our understanding of the convergence rate of BB method
is still very limited.
In fact, the known convergence rate of BB method \cite{raydan1993barzilai,dai2002r}, which is roughly $\left(1-\frac{1}{\kappa}\right)^{1/n}$,
is much worse than the rate of SD method ($ 1 - \frac{2 }{\kappa+1}$, where $\kappa$ is the condition number). 
On one hand, this is understandable since in earlier days, researchers are very much interested in the qualitative behavior (linear convergence v.s. superlinear convergence), rather than the specific  convergence rate. 
On the other hand, there has been a recent trend of analyzing the rate of convergence in large-scale optimization (see, e.g., \cite{nestrov12,johnson2013accelerating}), since it is believed
that the specific rate can help us better understand these algorithms.
In this context, the current gap between the rate of \cite{dai2002r} 
and SD method is not desirable:
considering the excellent practical performance of BB method,
there is a huge discrepancy between theory and practice.
How far can we push the boundary of the theory? 
 
We conjectured that BB method has at least the same rate of convergence as SD methods. Although the conjecture seems not surprising, there
are plenty of examples in optimization area that
a seemingly fast method has very slow worst-case convergence rate.
A classical example is the simplex method
which was shown to have exponential time worst-case complexity \cite{zadeh2009worst} \footnote{More rigorously speaking,
for most common pivoting rules the worst-case
time is exponential, and there is no known (deterministic) variant
that has polynomial time.}.
A more recent example is the cyclic coordinate descent method which was found to have a rate of convergence that can be $O(n)$ times worse than gradient descent methods in the worst case \cite{sun2019worst}.
Similarly, it was unclear a priori whether BB method 
has a bad worst-case convergence rate.

In this paper, we prove that
the convergence rate of the BB method is at least as 
good as the SD method in the sense of $R$-linear convergence, for quadratic problems. In particular, we prove that for $n$-dimensional strongly convex quadratic functions, the BB method converges $R$-linearly with a rate of $1-\frac{1}{\kappa}$. This rate is indeed comparable to that of the SD method. Moreover, we show that the lower bound of the convergence rate, if we select special (degenerate) initial points, is exactly $1-\frac{1}{\kappa}$ by giving an example in the $n$-dimensional case. This finding indicates that our convergence rate cannot be further improved without additional assumptions. 

The outline of this paper is as follows. We first present the notations and introduce the BB method in Section \ref{sec:setting}. In Section \ref{sec:motivation} we provide some intuition about how the BB method proceeds during iterations. This process motivates the proof idea of categorizing modes in our main theorem, which is presented in Section \ref{sec:theorem}. The lower bound example, as well as some discussions, are provided in Section \ref{sec:lower-bound}. Finally, the conclusion is made in Section \ref{sec:conclusion}.

\section{The BB Method for the Quadratic Case}\label{sec:setting}

In this paper, we focus on the BB method for the $n$-dimensional quadratic case. Consider the following quadratic optimization problem:
\begin{equation}\label{eq:problem}
\underset{x\in\mathbb{R}^n}{\min}~f(x)=\frac{1}{2}x^\top Ax-c^\top x
\end{equation}
where $A\in \mathcal{S}^n_{++}$. The classic steepest descent (SD) method takes the negative gradient as the search direction and chooses the stepsize as the minimizer along the search direction. Denote $g_k=\nabla f(x_k)$, then the SD method can be represented as 
\begin{equation*}
    x_{k+1}=x_k-\alpha_k^{SD}g_k~~ \text{where}~~ \alpha_k^{SD}=\arg \underset{\alpha>0}{\min}~f(x_k-\alpha g_k).
\end{equation*}

The BB method takes the same search direction as the SD method. Nevertheless, it gives a two-point stepsize rule to automatically determine the stepsize. In particular, the stepsize $\alpha_k$ is selected such that $\alpha_k^{-1}I$ best approximates the Hessian $f$. Denote $s_{k-1}=x_k-x_{k-1}$ and $y_{k-1}=g_k-g_{k-1}$. Through an approach similar to the quasi-Newton method, the BB method solves the following problem
\begin{equation*}
    \arg \underset{\alpha>0}{\min}~\|\alpha^{-1}s_{k-1}-y_{k-1}\|_2,
\end{equation*}
which yields
\begin{equation*}
\alpha_k^{BB}=\frac{s_{k-1}^\top s_{k-1}}{s_{k-1}^\top y_{k-1}}.
\end{equation*}

In the following, we derive some concrete iterative relations for the quadratic case. Without losing generality, we assume $c=0$ since the BB method is invariant under translations. In this case, $g_k=Ax_k$. Therefore, we have 
\begin{equation*}
    s_{k-1}=x_k-x_{k-1}=\alpha_{k-1}g_{k-1},~~ y_{k-1}=g_k-g_{k-1}=A(x_k-x_{k-1})=\alpha_{k-1}A g_{k-1}.
\end{equation*}
Thus, the BB stepsizes $\alpha_k$ can be represented by $g_{k-1}$ and $A$:
\begin{equation*}
\alpha_k^{BB}=\frac{g_{k-1}^\top g_{k-1}}{g_{k-1}^\top Ag_{k-1}}.
\end{equation*}
We can further obtain the relation among the gradients $g_{k+1}, g_k$ and $g_{k-1}$ by:
\begin{equation}\label{eq:gradient}
g_{k+1}=Ax_{k+1}=Ax_k-\alpha_k^{BB}Ag_k=g_k-\frac{g_{k-1}^\top g_{k-1}}{g_{k-1}^\top Ag_{k-1}}Ag_k.
\end{equation}

We now compute the eigenvalue decomposition of relation \eqref{eq:gradient}. Without losing generality, assume that $A$ has eigenvalues $0<\lambda_1\leq\cdots\leq\lambda_n$ with the corresponding eigenvectors $v_1, \cdots, v_n$. Note that $v_1, \cdots, v_n$ form an orthogonal unit basis in $\mathbb{R}^n$, so $g_k$ is uniquely decomposed into $g_k=\sum_{i=1}^nd^i_kv_i$ where $d_k^1, \cdots, d_k^n$ are the coefficients. Thus we can also decompose \eqref{eq:gradient} into 
\begin{equation}\label{eq:decomposition}
    \begin{aligned}
        \sum_{i=1}^nd^i_{k+1}v_i&=\sum_{i=1}^nd^i_kv_i-\frac{(\sum_{j=1}^nd^j_{k-1}v_j)^\top (\sum_{j=1}^nd^j_{k-1}v_j)}{(\sum_{j=1}^nd^j_{k-1}v_j)^\top A(\sum_{j=1}^nd^j_{k-1}v_j)}\cdot A\sum_{i=1}^nd^i_kv_i\\
        &=\sum_{i=1}^nd^i_kv_i-\frac{(\sum_{j=1}^nd^j_{k-1}v_j)^\top (\sum_{j=1}^nd^j_{k-1}v_j)}{(\sum_{j=1}^nd^j_{k-1}v_j)^\top (\sum_{j=1}^nd^j_{k-1}\lambda_jv_j)}\cdot \sum_{i=1}^nd^i_k\lambda_iv_i\\
        &=\sum_{i=1}^nd^i_kv_i-\frac{\sum_{j=1}^n(d^j_{k-1})^2}{\sum_{j=1}^n\lambda_j(d^j_{k-1})^2}\cdot \sum_{i=1}^nd^i_k\lambda_iv_i.
    \end{aligned}
\end{equation}
Comparing the coefficients before each $v_i$ in (\ref{eq:decomposition}), we obtain
\begin{equation}\label{eq:coefficient}
d^i_{k+1}=d^i_k\cdot\left(\frac{\sum_{j=1}^n(\lambda_j-\lambda_i)(d^j_{k-1})^2} {\sum_{j=1}^n\lambda_j(d^j_{k-1})^2}\right), ~~i=1, \cdots, n.
\end{equation}
Specifically, in the first iteration we perform a step of the gradient descent method, yielding
\begin{equation}\label{eq:first-iter}
d^i_1=d^i_0\cdot\left(\frac{\sum_{j=1}^n(\lambda_j-\lambda_i)(d^j_0)^2} {\sum_{j=1}^n\lambda_j(d^j_0)^2}\right), ~~i=1, \cdots, n.
\end{equation}
Relation \eqref{eq:coefficient} gives the dynamics of the coefficients of each eigenvector direction as the BB method proceeds. From now on we will focus on analyzing this dynamics.

\section{Dynamics of the Coefficients}\label{sec:motivation}

In this section, by working out a simple example, we make some observations to provide an intuition about how the dynamics of the coefficients proceeds during the BB iteration.

Consider a $4$-dimensional quadratic minimization problem $\min_{x\in\mathbb{R}^4}\frac{1}{2}x^\top Ax$, where $A$ has eigenvalues $(\lambda_1, \lambda_2, \lambda_3, \lambda_4) = (0.001, 0.01, 0.1, 1)$, and the initial point is chosen subject to a uniform distribution between $[0, 1]$. The following figure is an example of the trajectories of the absolute value of the coefficients $|d_k^j|$ for $j=1, 2, 3, 4$ in the BB iteration\footnote{Although the trajectories can be different for some specific cases, this figure is a typical characterization of the trajectories in regard to the random initial point selection.}.

\begin{figure}[H]\label{fig:coef}
	\centering
    \includegraphics[width=6cm]{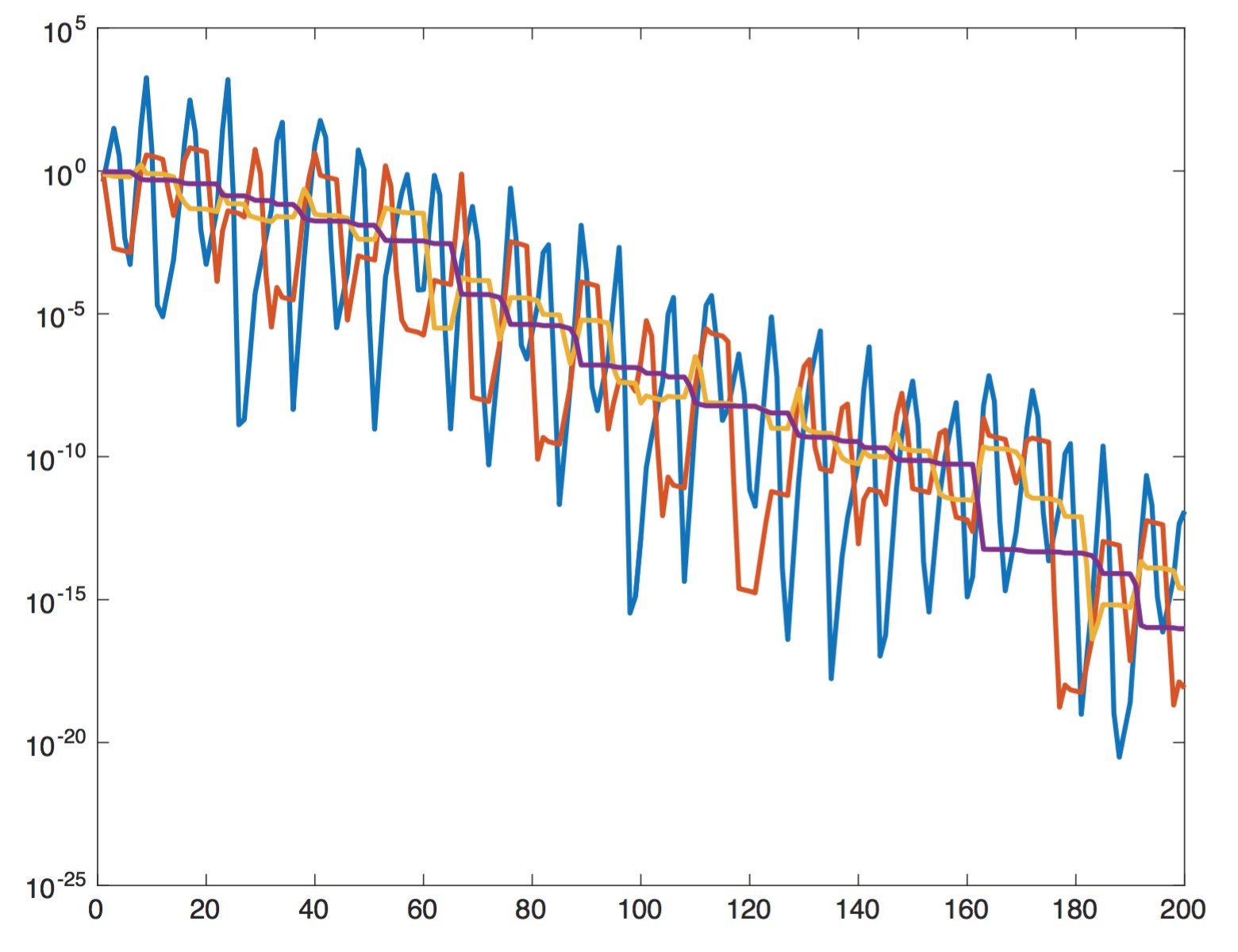}
    \caption{The trajectories of the absolute value of the coefficients $|d_k^j|$ in a $4$-dimensional BB iteration.}
\end{figure}

The purple, orange, red and blue lines in Figure \ref{fig:coef} correspond to the trajectories of the absolute value of the coefficients $|d_k^j|$ for $j=1, 2, 3, 4$ (i.e., from the minimal eigenvalue to the maximal one) separately. We have the following observations from Figure \ref{fig:coef}:
\begin{enumerate}
	\item The trajectory corresponding to the minimal eigenvalue (the purple line) always descends, whereas the others fluctuate near it instead of always staying above or below.
	\item If one trajectory is significantly higher than the others at some iteration, it will drop dramatically within two iterations.
\end{enumerate}

These observations provide an intuition about how the coefficients change during the BB iteration. Except the one corresponding to the minimal eigenvalue, all coefficients have two modes: the shrinking mode (when the absolute value of the coefficient decreases) and the fluctuation mode (when the absolute value of the coefficient increases). The switch between the two modes causes the BB method to have non-monotone behaviors. However, if one coefficient has a significantly larger absolute value than the others at some iteration, it must switch from the fluctuation mode to the shrinking mode in the next. This implies that the coefficients 
alternate between fluctuation and trending in a downward pattern, which motivates our proof idea.

\section{Main Results}\label{sec:theorem}
In this section, we formally define the ``shrinking mode'' and the ``fluctuation mode'' following the motivation given in the previous section. Then with this definition, we make an surprisingly simple analysis to prove that the BB method converges linearly with a rate comparable to the classic SD method 
\footnote{Our first version of the proof takes
at least 6 pages, based on detailed analysis of multiple cases.
We later greatly simplified the proof to just 2.5 pages.}

We first provide a formal definition of the ``shrinking mode'' and the ``fluctuation mode'':

\begin{definition}[Shrinking mode and fluctuation mode]
    ~
    \begin{enumerate}
        \item For a given index $i$, the sequence $\{d_k^i\}$ is in the shrinking mode at iteration $k$ if
        \begin{equation}\label{eq:def-shrinking}
            \sum_{j=1}^n(\lambda_j-\lambda_i)(d_k^j)^2\geqslant0;
        \end{equation}
        \item For a given index $i$, the sequence $\{d_k^i\}$ is in the fluctuation mode at iteration $k$ if
        \begin{equation}\label{eq:def-fluctuation}
        \sum_{j=1}^n(\lambda_j-\lambda_i)(d_k^j)^2<0.
        \end{equation}
    \end{enumerate}
    \end{definition}

According to relation \eqref{eq:coefficient}, the mode of $d_{k-1}^i$ decides whether $d_{k+1}^i$ and $d_k^i$ share the same sign, which is seemingly not related to ``shrinking'' or ``fluctuation''. Nevertheless, the mode actually determines whether the absolute value of $d_{k+1}^i$ can be greater than that of $d_k^i$. In fact, if $d_{k-1}^i$ is in the shrinking mode, $\left|d_{k+1}^i\right|$ is always smaller than $|d_k^i|$. On the other hand, if $d_{k-1}^i$ is in the fluctuation mode, $\left|d_{k+1}^i\right|$ can increase. However, as long as $\left|d_{k-1}^i\right|$ is large enough, we still have $\left|d_{k+1}^i|<|d_k^i\right|$. This finding is illustrated by the following proposition:

\begin{proposition}\label{prop:shrink}
    Consider the BB method for solving the quadratic minimization problem (\ref{eq:problem}) and let $\left\{d_k^i\right\}$ be the coefficient sequences. Then, it holds that
    \begin{equation}
    \label{eq::iterbound_general}
        \left|d_{k+1}^i\right|\leq\left|d_k^i\right|\cdot\left(\max\left\{\frac{\lambda_i}{\lambda_1}-1, 1-\frac{\lambda_i}{\lambda_n}\right\}\right), \forall k\geq 1, 1\leq i\leq n.
    \end{equation}
    Moreover, if at some iteration, one of the following conditions holds:
    \begin{enumerate}
        \item  $d_{k-1}^i$ is in the shrinking mode;
        \item  $d_{k-1}^i$ is in the fluctuation mode, and $(d_{k-1}^i)^2\geq\sum_{j=1}^{i-1}(d_{k-1}^j)^2$;
    \end{enumerate}
    then $d_{k+1}^i$ and $d_k^i$ satisfy
    \begin{equation}
    \label{eq::iterbound_shrink}
        \left|d_{k+1}^i\right|\leq \left(1-\frac{1}{\kappa}\right)\left|d_k^i\right|,
    \end{equation}
    where $\kappa=\frac{\lambda_n}{\lambda_1}$ is the condition number of $A$.
    \end{proposition}

\begin{proof}
Note that if $d_k^i=0$, then $d_k^{j}=0$ holds for all $j\geq i$. Thus, without losing generality, assume that $d_k^i\neq0$. 

We first prove \eqref{eq::iterbound_general}. According to \eqref{eq:coefficient},
{\footnotesize 
\begin{equation*}
\begin{aligned}
    \left|\frac{d_k^i}{d_{k-1}^i}\right|&=\left|\frac{\sum_{j=1}^n(\lambda_j-\lambda_i)(d^j_{k-1})^2} {\sum_{j=1}^n\lambda_j(d^j_{k-1})^2}\right|\\
    &\leq\max\left\{\frac{\sum_{j=1}^{i-1}(\lambda_i-\lambda_j)(d_{k-1}^j)^2}{\sum_{j=1}^n\lambda_j(d_{k-1}^j)^2}, \frac{\sum_{j=i+1}^n(\lambda_j-\lambda_i)(d_{k-1}^j)^2}{\sum_{j=1}^n\lambda_j(d_{k-1}^j)^2}\right\}\\
    &=\max\left\{\sum_{j=1}^{i-1}\left(\frac{\lambda_i-1}{\lambda_j}\right)\cdot\frac{\lambda_j(d_{k-1}^j)^2}{\sum_{j=1}^n\lambda_j(d_{k-1}^j)^2}, \sum_{j=i+1}^n\left(1-\frac{\lambda_i}{\lambda_j}\right)\cdot\frac{\lambda_j(d_{k-1}^j)^2}{\sum_{j=1}^n\lambda_j(d_{k-1}^j)^2}\right\}\\
    &\leq\max\left\{\left(\frac{\lambda_i}{\lambda_1}-1\right)\cdot\sum_{j=1}^{i-1} \frac{\lambda_j(d_{k-1}^j)^2}{\sum_{j=1}^n\lambda_j(d_{k-1}^j)^2}, \left(1-\frac{\lambda_i}{\lambda_n}\right)\cdot\sum_{j=i+1}^n \frac{\lambda_j(d_{k-1}^j)^2}{\sum_{j=1}^n\lambda_j(d_{k-1}^j)^2}\right\}\\
    &\leq\max\left\{\frac{\lambda_i}{\lambda_1}-1, 1-\frac{\lambda_i}{\lambda_n}\right\}.
\end{aligned}
\end{equation*}
}

Next, to prove \eqref{eq::iterbound_shrink}, we discuss under the two conditions separately. 

If $d_{k-1}^i$ satisfies the first condition, we have $\sum_{j=1}^{i-1}(\lambda_i-\lambda_j)(d_k^j)^2\leq\sum_{j=i+1}^n(\lambda_j-\lambda_i)(d_k^j)^2$ according to \eqref{eq:def-shrinking}. Therefore,
\begin{equation*}
    \begin{aligned}
   \left|\frac{d_{k+1}^i}{d_k^i}\right| &=
   \frac{\sum_{j=1}^n(\lambda_j-\lambda_i)(d^j_{k-1})^2} {\sum_{j=1}^n\lambda_j(d^j_{k-1})^2} 
        \leq \frac{\sum_{j=i+1}^n(\lambda_j-\lambda_i)(d_{k-1}^j)^2}{\sum_{j=1}^n\lambda_j(d_{k-1}^j)^2}  \\
        &= \sum_{j=i+1}^n
        \left(1-\frac{\lambda_i}{\lambda_j}\right)
        \cdot\frac{\lambda_j(d_{k-1}^j)^2}{\sum_{j=1}^n\lambda_j(d_{k-1}^j)^2}
        \leq  \left(1-\frac{\lambda_i}{\lambda_n}\right)\cdot\sum_{j=i+1}^n \frac{\lambda_j(d_{k-1}^j)^2}{\sum_{j=1}^n\lambda_j(d_{k-1}^j)^2}\\
        &\leq 1-\frac{\lambda_i}{\lambda_n}
        \leq 1-\frac{1}{\kappa}.
    \end{aligned}
\end{equation*}

If $d_{k-1}^i$ satisfies the second condition, we have $\sum_{j=1}^{i-1}(\lambda_i-\lambda_j)(d_k^j)^2>\sum_{j=i+1}^n(\lambda_j-\lambda_i)(d_k^j)^2$ according to \eqref{eq:def-fluctuation}. Therefore,
\begin{equation*}
    \begin{aligned}
        \left|\frac{d_{k+1}^i}{d_k^i}\right|&=\frac{\sum_{j=1}^n(\lambda_i-\lambda_j)(d^j_{k-1})^2} {\sum_{j=1}^n\lambda_j(d^j_{k-1})^2}
        \leq \frac{\sum_{j=1}^{i-1}(\lambda_i-\lambda_j)(d_{k-1}^j)^2}{\sum_{j=1}^n\lambda_j(d_{k-1}^j)^2}\\
        &\leq \frac{\sum_{j=1}^{i-1}(\lambda_i-\lambda_j)(d_{k-1}^j)^2}{\lambda_i(d_{k-1}^i)^2}
        \leq  \left(1-\frac{\lambda_1}{\lambda_i}\right)\cdot\sum_{j=1}^{i-1} \frac{(d_{k-1}^j)^2}{(d_{k-1}^i)^2}\\
        &\leq 1-\frac{\lambda_1}{\lambda_i}
        \leq 1-\frac{1}{\kappa}.
    \end{aligned}
\end{equation*}

Combining the two cases completes the proof.
\end{proof}

With Proposition \ref{prop:shrink}, we present our main theorem of linear convergence with a rather simple proof:
\begin{theorem}\label{thm:convergence}
	Consider the BB method for solving the quadratic minimization problem (\ref{eq:problem}). Let $\{d_k^i\}$ be the coefficient sequences and $\theta=1-\frac{1}{\kappa}$, where $\kappa=\frac{\lambda_n}{\lambda_1}$ is the condition number of $A$. Then $\{d_k^i\}$ converges at least linearly with a rate of $\theta$. In particular, for any $k\geq1$ and $i=1, \cdots, n$, $|d_k^i|\leq F_i\theta^k$, where $F_i$ is a constant defined by the following recursive sequence:
	\begin{equation}\label{eq:constant}
	\left\{\begin{aligned}
	F_1&=d_0^1\\
	F_i&=\max\left\{d_0^i, \frac{d_1^i}{\theta}, \theta^{-2} C^2 \sqrt{\sum_{j=1}^{i-1}F_j^2}\right\}
	\end{aligned}\right.
	\end{equation}
	where $C=\max\left\{\frac{\lambda_i}{\lambda_1}-1, 1-\frac{\lambda_i}{\lambda_n}\right\}.$
\end{theorem}

\begin{proof}
	We prove the result by induction on $i=1, \cdots, n$. For $i=1$, note that \eqref{eq:def-shrinking} always holds, so $d_k^1$ is always in the shrinking mode. According to Proposition \ref{prop:shrink}, we have $\left|\frac{d_{k+1}^i}{d_k^i}\right| \leq 1-\frac{\lambda_1}{\lambda_n}$ for any $k$. This implies that $\left|d_k^1\right|\leq \left|d_0^1\right|\theta^k$.
	
	Suppose that the result holds for all $d_k^j$ when $j=1, \cdots, i-1$. We prove by contrapositive that it also holds for all $d_k^i$. The result trivially holds for $k=0, 1$. Assume, in contrast, that the result does not hold for some $k\geq 2$. We find the minimal $k\geq 2$ such that $|d_k^i|>F_i\theta^k$. Note that by Proposition \ref{prop:shrink}, $\left|\frac{d_k^i}{d_{k-1}^i}\right|\leq C$ for any $k\geq 1$. Therefore, 
	\begin{equation*}
		\left|d_{k-2}^i\right|\geq\frac{\left|d_k^i\right|}{C^2} >\frac{F_i\theta^k}{C^2} \geqslant \sqrt{\sum_{j=1}^{i-1}F_j^2}\theta^{k-2} \geq\sqrt{\sum_{j=1}^{i-1}(d_{k-2}^j)^2}.
	\end{equation*}
	Using the result of Proposition \ref{prop:shrink}, we immediately obtain that $|d_{k-1}^i|\geq |d_k^i|/\theta>F_i\theta^{k-1}$, which contradicts our assumption that $k$ is the minimal index such that the result fails to hold. Hence we have shown that the result holds for all $d_k^i$, and by induction the proof is complete.
\end{proof}

\textbf{Remark 1}: The linear rate we achieve in Theorem \ref{thm:convergence} is $1-\frac{1}{\kappa}$, which is much faster than the previous results \cite{dai2002r}. Furthermore, this rate is comparable to the convergence rate of the SD method ($1-\frac{2}{\kappa+1}$), so our result indicates that the BB method is at least comparable to the SD method in terms of the convergence
rate. 
It is worth mentioning that the constant term $F_i$ in the theorem is highly related to the eigenvalues of $A$, and a trivial upper bound for $F_i$ is $\mathcal{O}((1+\kappa^2)^{i-1})$.

\textbf{Remark 2}: We briefly discuss the differences
of our proof with earlier proofs of convergence in \cite{raydan1993barzilai,dai2002r}. 

\section{Example with Exact Linear Rate Convergence}\label{sec:lower-bound}

One may wonder whether the BB method achieves an even faster convergence rate since in practice it performs much better than the SD method. The answer we provide to this question is: not always. In particular, we present an example with a specifically chosen initialization to show that the BB method may converge at the same speed as the SD method.

\begin{proposition}\label{prop:lower-bound}
	Consider the BB method for solving the quadratic minimization problem (\ref{eq:problem}). For any $A\in \mathcal{S}^n_{++}$, there exists $x_0\in \mathbb{R}^n$ such that the sequence $\{x_k\}$ generated by the BB method satisfies
	\begin{equation}\label{eq:lower-bound}
		|g_k|\geq\left(\frac{\kappa-1}{\kappa+1}\right)^k|g_0|,
    \end{equation}
    where $g_k=Ax_k-c$ is the gradient at each iteration.
\end{proposition}
\begin{proof}
    Assume that $A$ has eigenvalues $0<\lambda_1\leq\cdots\leq\lambda_n$ with the corresponding eigenvectors $v_1, \cdots, v_n$. Let $x_0=A^{-1}(c+v_1+v_n)$, and thus $g_0=Ax_0-c=v_1+v_n$. Decomposing $g_0$ according to the basis of $v_1, \cdots, v_n$ yields 
    \begin{equation*}
        d_0^1=d_0^n=1,~~ d_0^j=0,~~ j=2, \cdots, n-1.
    \end{equation*}
    
    By relation \eqref{eq:coefficient}, it is apparent to see that $d_k^j=0$ for all $j=2, \cdots, n-1$. Furthermore, for all $k\geq 1$,
    \begin{equation}\label{eq:example}
        \begin{aligned}
            d_{k+1}^1&=d_k^1\cdot \left(\frac{\sum_{j=1}^n(\lambda_j-\lambda_1)(d^j_{k-1})^2} {\sum_{j=1}^n\lambda_j(d^j_{k-1})^2}\right)=d_k^1\cdot\frac{(\lambda_n-\lambda_1)(d^n_{k-1})^2}{\lambda_1(d^1_{k-1})^2+\lambda_n(d^n_{k-1})^2},\\
            d_{k+1}^n&=d_k^1\cdot \left(\frac{\sum_{j=1}^n(\lambda_j-\lambda_n)(d^j_{k-1})^2} {\sum_{j=1}^n\lambda_j(d^j_{k-1})^2}\right)=d_k^n\cdot\frac{(\lambda_1-\lambda_n)(d^1_{k-1})^2}{\lambda_1(d^1_{k-1})^2+\lambda_n(d^n_{k-1})^2}.\\
        \end{aligned}
	\end{equation}
    Note that $d_0^1=d_0^n=1$ and $d_1^1=\frac{\lambda_n-\lambda_1}{\lambda_n+\lambda_1}, d_1^n=-\frac{\lambda_n-\lambda_1}{\lambda_n+\lambda_1}$ due to \eqref{eq:first-iter}, so $(d_0^1)^2=(d_0^n)^2, (d_1^1)^2=(d_1^n)^2$. According to relation \eqref{eq:example}, we can show by induction that $(d_k^1)^2=(d_k^n)^2$ for every $k\geq0$. As a result, \eqref{eq:example} is simplified as
    \begin{equation*}
            d_{k+1}^1=d_k^1\cdot \frac{\lambda_n-\lambda_1}{\lambda_n+\lambda_1},~~
            d_{k+1}^n=d_k^n\cdot \frac{\lambda_1-\lambda_n}{\lambda_n+\lambda_1},
    \end{equation*}
    which implies that 
    \begin{equation*}
        d_1^k=\left(\frac{\lambda_n-\lambda_1}{\lambda_n+\lambda_1}\right)^k=\left(\frac{\kappa-1}{\kappa+1}\right)^k,~~ d_n^k=\left(\frac{\lambda_1-\lambda_n}{\lambda_n+\lambda_1}\right)^k=(-1)^k\left(\frac{\kappa-1}{\kappa+1}\right)^k.
    \end{equation*}
    Therefore inequality \eqref{eq:lower-bound} holds for the example we construct.
\end{proof}

A $2$-dimensional version of the example is mentioned in \cite{dai2013new}. Specifically, \cite{dai2013new} proved that the BB method has a superlinear convergence rate for probability one measure of initial points while converges linearly for zero measure of initial points. Proposition \ref{prop:lower-bound} demonstrates that for any-dimensional quadratic problems, the BB method converges no faster than the SD method when selecting some specific initial points. It also implies that there is no superlinear rate convergence for the BB method if we allow arbitrary initialization, which indicates the tightness of Theorem \ref{thm:convergence}. Nevertheless, we believe that the set of such initial points is zero measure. As for generic initialization, just as the case when $n=2$, the BB method can still be much faster for generic initialization; 
we leave the investigation of this matter to future work. 

\section{Conclusion}\label{sec:conclusion}

In this paper, we consider the BB method for solving  quadratic problems. 
The existing rate of convergence is much worse
than the steepest descent method (or gradient descent with constant stepsize).
We prove that the convergence rate of BB methods
is at least $1-\frac{1}{\kappa}$, which is stronger than the 
existing results and is comparable to the SD method. 
Moreover, an example is constructed to show that for some initial points, the BB method converges at the same rate as the SD method.
An interesting research direction is to study whether a faster convergence rate exists when additional assumptions are imposed on the initial points. 


\end{document}